\documentclass{amsart}
\usepackage{amssymb}
\usepackage{mathtools}
\usepackage{graphicx}\usepackage{amsfonts}
\usepackage{enumerate}

\def\bD{\boldsymbol\Delta}
\def\R{{\mathbb R}}
\def\C{{\mathbb C}}

\def\Z{{\mathbb Z}}

\def\<{\langle}
\def\>{\rangle}

\def\E{\mathbb E}
\def\D{\mathcal D}
\def\Le{\mathcal L}
\def\Id{\mathcal I}
\def\Ri{{\mathcal R}}
\def\eps{\epsilon}
\def\ep{\underline{\epsilon}}

\def\0{\underline 0}

\def\r{\underline r}
\def\e{\underline e}
\def\X{\underline X}
\def\1{\underline 1}
\def\c{\underline c}

\def\x{\underline x}
\def\y{\underline y}
\def\w{\underline w}

\def\u{\underline u}
\def\Y{\underline Y}
\def\var{\mathbb{V}\mathrm{ar}}
\def\cov{\mathbb{C}\mathrm{ov}}

\newcommand{\bel}{\begin{equation}\label}
\newcommand{\ee}{\end{equation}}
      \newtheorem{theorem}{Theorem}[section]

       \newtheorem{lemma}[theorem]{Lemma}
       \newtheorem{remark}{Remark}[section]
\theoremstyle{definition}

\begin{document}
\title[Recursion for optimal estimators under rotation] {Exploring recursion for optimal estimators under cascade rotation}
\author{Jan Kowalski}
\address{Politechnika Warszawska, Warszawa, Poland}
\author{Jacek Weso\l owski}
\address{G\l \'owny Urząd Statystyczny and  Politechnika Warszawska, Warszawa, Poland}
\email{wesolo@mini.pw.edu.pl}

\begin{abstract}
We are concerned with optimal linear estimation of means on subsequent occasions under sample rotation where evolution of samples in time is designed through a cascade pattern. It has been known since the seminal paper of Patterson (1950) that when the units are not allowed to return to the sample after leaving it for certain period (there are no gaps in the rotation pattern), one step recursion for optimal estimator holds. However, in some important real surveys, e.g. Current Population Survey in the US or Labour Force Survey in many countries in Europe, units return to the sample after being absent in the sample for several occasions (there are gaps in rotation patterns). In such situations difficulty of the question of the form of the recurrence for optimal estimator increases drastically. This issue has not been resolved yet. Instead alternative sub-optimal approaches were developed, as $K$-composite estimation (see e.g. Hansen et al. (1955)), $AK$-composite estimation (see e.g. Gurney and Daly (1965) or time series approach (see e.g. Binder and Hidiroglou (1988)).

In the present paper we overcome this long-standing difficulty, that is, we present analytical recursion formulas for the optimal linear estimator of the mean for schemes with gaps in rotation patterns. It is achieved under some technical conditions: ASSUMPTION I and ASSUMPTION II (numerical experiments suggest that these assumptions might be universally satisfied).  To attain the goal we develop an algebraic operator approach  which allows to reduce the problem of recursion for the optimal linear estimator to two issues: (1) localization of roots (possibly complex) of a polynomial $Q_p$ defined in terms of the rotation  pattern ($Q_p$ happens to be conveniently expressed through Chebyshev polynomials of the first kind), (2) rank of a matrix $S$ defined in terms of the rotation pattern and the roots of the polynomial $Q_p$. In particular, it is shown that the order of the recursion is equal to one plus the size of the largest gap in the rotation pattern. Exact formulas for calculation of the recurrence coefficients are given - of course, to use them one has to check (in many cases, numerically) that ASSUMPTIONs I and II are satisfied. The solution is illustrated through several examples of rotation schemes arising in real surveys.
\end{abstract}

\maketitle

\section{Introduction}

Repeated surveys with rotation of elements in samples are commonly used by
statistical offices and other institutions. Predesigned rotation of (groups of)
elements in a form of cascade patterns, that is such schemes when, on
each occasion the 'oldest' element (group of elements) leaves the sample and is replaced
by a new one, is also very popular but information carried in the survey
data is often not exploited in full.   This in turn leads to
constructing sub-optimal estimators with variance above the
achievable minimum. To enhance the use of optimal estimators in rotation schemes, in the seminal paper Patterson (1950) introduced the idea of
recurrence for best linear unbiased estimators (BLUEs) of the mean on each occasion. The main assumptions
were that the unknown population means are deterministic
and the responses are random variables whose variances and
correlation structure are fully known. Under exponential correlation and
assuming further that any element leaving the sample does not return
to the survey, Patterson proved that for any occasion $t$ the BLUE $\hat{\mu}_t$ of the
current mean $\mu_t$ at time $t$ (based on all past observations) can be computed from the following one-step recurrence:
\begin{equation}\label{eqn_rekur_Pat}
    \hat{\mu}_t = a_1(t) \hat{\mu}_{t-1} +
    \r^T_0(t) \X_{t} +
    \r^T_1(t) \X_{t-1}
\end{equation}
where $\X_i$ is the vector of observations at time $i=t,t-1$. The formulas for the
recurrence coefficients, that is the numbers $a_1(t)$ and the vectors $\r_0(t)$,
$\r_1(t)$, were given there as well. (Here and throughout the paper a vector, say $\r$, is understood as a column, $\r^T$ is its transpose. For two vectors $\r=(r_1,\ldots,r_n),\;\w=(w_1,\ldots,w_n)\in\R^n$ the expression $\r^T\w=\sum_{i=1}^n\,r_iw_i$ is just the scalar product of $\r$ and $\w$.)

Patterson's assumption that {\bf a unit leaving a sample never returns to the survey} was a core of his approach. If this assumption is violated (that is, there are gaps in rotation patterns) it has been known for years that serious difficulties arise if one seeks an analogue of the recurrence \eqref{eqn_rekur_Pat}. Being aware of this (see, e.g. Yansaneh and Fuller, 1998) researchers rather tried alternative approaches: Classical $K$-composite estimator was proposed in Hansen et al. (1955). Its optimality properties were developed in Rao and Graham (1964) and more recently in Ciepiela et al. (2012). The main difference is that instead of seeking the recurrence for BLUE, these authors restrict the optimality issue to linear unbiased estimators satisfying just the first order recurrence, that is the variance of the estimator based on the most recent estimator and observations from the last two occasions only is minimized. Adjustments, known as  $AK$-composite estimator, introduced in Gurney and Daly (1965), have been  developed, e.g. in Cantwell (1988, 1990) and  Cantwell and Caldwell (1998) - actually in these papers the authors introduce the notion of balanced multi-level design, and one-level design corresponds to the cascade pattern we consider here. Another approach based on regression composite estimator has been considered in Bell (2001), Fuller and Rao (2001) and Singh et al. (2001) (with implications for Canadian Labour Force Survey).

The difficulty in recursive estimation in repeated surveys for patterns with gaps was raised in Yansaneh and Fuller (1998), who analyzed variances of composite estimators in several rotation schemes. For a relatively current description of the state of art in the area one can consult Steel and McLaren (2008), in particular Sec. IV on different rotation patterns and Sec. V on composite estimators. Comparisons of effectiveness under different cascade patterns can be found in McLaren and Steel (2000) and Steel and McLaren (2002). A very recent paper on optimal estimation under rotation is by Towihidi and Namazi-Rad (2010). Some of these references deal also with time series approach (which is not considered in this paper) in which the unknown means are treated as random quantities - an overview of such approach can be found in Binder and Hidiroglou (1988). For a more recent development of this setting see e.g. Lind (2005).

As for the original approach of Patterson, the next result
concerning the recursive form of the BLUE was presented in
Kowalski (2009), where singleton gaps in the rotation pattern were allowed. As in Patterson (1950), this paper was devoted to the "classical" situation in which the coefficients in \eqref{eqn_rekur} below are allowed to depend on $t$.   Three conclusions
from that work have an impact on this paper. Firstly, it was
suggested that the formula (\ref{eqn_rekur_Pat}) may be generalized
to an arbitrary rotation scheme (including gaps in the pattern) by incorporating the
optimal estimators and observations from a probably larger (but still as small as possible) number of past
occasions and that the order of the recurrence should depend on the
size of the largest gap. Secondly, it was observed there that the
exponential correlation, as assumed in Patterson (1950), is crucial for
obtaining the recursive representation and that it is plausible to
restrict oneself to the class of 'cascade' schemes. Both these assumptions are kept below.
Finally, since according to numerical simulations the recurrence
coefficients appear to be quickly convergent as $t \rightarrow
\infty$, a suggestion was made to consider the 'limiting' case of
the "classical" setting, in which the recurrence coefficients do not change in
time.

We want to stress that in the present paper {\bf any set of gaps in the cascade rotation pattern is allowed}.
The aim is to show that the recurrence
\begin{equation}\label{eqn_rekur}
    \hat{\mu}_t = a_1 \hat{\mu}_{t-1} + \ldots + a_p \hat{\mu}_{t-p}
    + \underline{r}_0^T\,\underline{X}_{t} +
    \underline{r}_1^T\,\underline{X}_{t-1} + \ldots +
    \underline{r}_p^T\,\underline{X}_{t-p}
\end{equation}
holds for any cascade rotation scheme and to find the order of recurrence $p$, the numerical coefficients $a_1,\ldots,a_p$ and the vector coefficients  $\r_0,\ldots,\r_p$. Let us emphasize that the representation \eqref{eqn_rekur} is
"stationary" in the sense that neither the order of the recurrence $p$
nor the recurrence coefficients $(a_i)$ and $(\r_i)$ depend
on $t$.

Our main result lies in reducing the recurrence problem to analysis of a certain polynomial $Q_p$ (of degree $p$, where $p-1$ is the size of the largest gap in the rotation pattern) and to the question of unique solvability of a certain linear system of
equations, which depends on roots of $Q_p$. Luckily the polynomial $Q_p$ happens to be conveniently expressed through Chebyshev polynomials of the first kind.
We provide a sufficient condition in terms of localization properties of  roots of $Q_p$  for existence of the
recursive form of the BLUE of order $p$, as given in \eqref{eqn_rekur}, and derive explicit formulas (exploiting roots of $Q_p$) for the recurrence coefficients $(a_i)$ and $(\r_i)$. The forms of the coefficients depend also on the unique solution of the linear system mentioned above. The result is illustrated by several examples related to the real life surveys.

The convergence of recursion coefficients which we observed numerically in many "classical" schemes (that is, with coefficients in the analogue of \eqref{eqn_rekur} depending on $t$) of different complexity indicates that
solution to such "stationary" recurrence problem should exist universally (actually only in the Patterson case, $p=1$, such convergence is formally proved).
If so it can be treated as an approximate solution for the
"classical" scheme. As the reader will see, this intuition is largely
confirmed in this paper. Our main result still is not universal even within models with exponential correlation. Our approach heavily relies on two assumptions (ASSUMPTION I and ASSUMPTION II below) which allow us to claim that the recurrence \eqref{eqn_rekur} holds true. Nevertheless, we performed many numerical experiments for different rotation patterns and different values of the correlation and they all suggest that both these assumptions may be universally satisfied. Unfortunately, at the present stage we are unable to confirm theoretically these observations.

The plan of the paper is as follows. In Section \ref{model} we introduce in mathematical terms the model we are working with. In Section
\ref{recurrence} we introduce our two core assumptions and formulate the main result of the paper. Section \ref{eg} contains examples of applications of the main result in several popular rotation schemes. The main body of mathematics is deferred to Appendix. In its first part \ref{App1}  algebraic properties of shift operators are considered. They are essential for the proof of the recursion formula which is given in the second part \ref{App2} of Appendix.

\section{Model}\label{model}
Let $(X_{i,j})_{i,j\in \Z}$ be a doubly infinite matrix of random variables. Heuristically, $X_{i,j}$ represents the value of variable $\mathcal{X}$ measured for the unit (rotation group) $i$ on the occasion $j$. We assume that the expectation of $X_{i,j}$ depends only on the occasion and not on the unit, that is
$$
\E\,X_{i,j}=\mu_j,\qquad\qquad \forall\,i,j\in\Z.
$$
Moreover, we assume exponential in time correlations between $X_{i,j}$'s for the same unit and no correlations between different units (following Patterson (1950) model), that is
$$\cov(X_{i,j},\,X_{k,l})=\rho^{|j-l|}\delta_{i,k}\qquad\qquad \forall\,i,j,k,l\in\Z,$$
where $|\rho|\in (0,1)$ and $\delta_{i,k}=1$ if $i=k$, otherwise $\delta_{i,k}=0$. (In practical situations often $\rho$ is in $[0,1)$. In the case $\rho=0$ observations from the past cannot improve present linear estimator of the mean, therefore we do not consider such case below.)  Consequently,  $$\var\,X_{i,j}=1,\qquad\qquad i,j\in \Z.$$

For any $j\in\Z$ we are interested in the BLUE of $\mu_j$ based on all available observations from occasions $i\le j$. For a fixed positive integer $N$ denote by
$$
\X_j=(X_{j,j},\,X_{j+1,j},\ldots,X_{j+N-1,j})^T
$$
the {\bf maximal sample} (of size $N$) on the occasion $j\in \Z$. Then
$$
\E\,\X_j=\mu_j\,\1,\qquad\qquad j\in \Z,
$$
where $\1=(1,1,\ldots,1)^T\in\R^N$, and
$$
\cov(\X_j,\,\X_{j-k})={\bf C}^k=\left[\cov(\X_j,\,\X_{j+k})\right]^T,\qquad\qquad j\in\Z,\;k\ge 0,
$$
where ${\bf C}$ is an $N\times N$ matrix of the form
$${\bf C} = \left[ \begin{array}{c c c c}
      0 & \rho & & 0 \\
      0 & \ddots & \ddots & \vspace{0pt} \\
       & \ddots & \ddots & \rho \vspace{2pt} \\
      0 & & 0 & 0
      \end{array} \right]$$
Note that ${\bf C}^n=0$ for any $n\ge N$.

The effective sample will be defined by a {\bf cascade pattern}, which is a vector $\ep=(\eps_1,\ldots,\eps_N)^T\in\{0,1\}^N$ with $\eps_1=\eps_N=1$. Let
$$
n=\sum_{j=1}^N\,\eps_j\qquad\mbox{and}\qquad h=N-n.
$$
Let $H$ be the set of zeros in the pattern $\ep$, that is $j\in H$ iff $\eps_j=0$. Obviously, $\#\,H=h$. A gap of size $m$ is a maximal set of sequential $m$ zeros, that is a set satisfying
$$
\{j,j+1,\ldots,j+m-1\}\subset H\qquad \mbox{and}\qquad j-1,\,j+m\not\in H.
$$
Consequently, $H$ is a union of, say, $s$ gaps of sizes $m_r$, $r=1,2,\ldots,s$, and $\sum_{r=1}^s\,m_r=h$.

The coverage $p$ of the pattern (see Kowalski, 2009 for equivalent definition) is the size of the largest gap increased by one:
$$
p=1+\max_{1\le r\le s}\,m_r.
$$

On each occasion $j\in\Z$ we may not observe the maximal sample $\X_j$ but the {\bf effective sample} of size $n$ defined by the cascade pattern $\ep$, that is the vector
$$
\Y_j=(X_{j+k-1,j},\,k\in\{1,\ldots,N\}\setminus H)^T,
$$
that is values of $X_{i,j}$'s represented by zeros (gaps) in the cascade pattern $\eps$ are removed from the sample.

We consider BLUE $\hat{\mu}_t$ of the mean $\mu_t$ on the occasion $t\in\Z$ which is based on observations $\Y_j$, $j\le t$. That is
$$
\hat{\mu}_t=\sum_{i=0}^{\infty}\,\tilde{\w}_i^T\,\Y_{t-i}
$$
with $\tilde{\w}_i\in\R^n$, $i\ge 0$, which minimize $\var\,\hat{\mu}_t$ under the unbiasedness constraints
$$
\tilde{\w}_0^T\,\1=1\qquad\qquad\mbox{and}\qquad\qquad \tilde{\w}_i^T\,\1=0,\quad i\ge 1.
$$

It is both obvious and crucial for our approach that, equivalently,
\bel{estyma}
\hat{\mu}_t=\sum_{i=0}^{\infty}\,\w_i^T\,\X_{t-i}
\ee
with $\w_i\in\R^N$, $i\ge 0$, minimizing $\var\,\hat{\mu}_t$ under unbiasedness constraints
\bel{unb}
\w_0^T\,\1=1,\qquad\qquad \w_i^T\,\1=0,\quad i\ge 1,
\ee
and cascade pattern constraints
\bel{pat}
\w_i^T\,\e_j=0\qquad \forall\,i\ge 0,\;\forall\,j\in H,
\ee
where $\e_j=(0,\ldots,0,1,0,\ldots,0)^T$ (with 1 at $j$th position) is $j$th vector of the canonical basis in $\R^N$, $j\in H$. Note that the constraint \eqref{pat} actually says that $j$th entries ($j\in H$) of vectors $\w_i$, $i\ge 0$, are all zeros.
\section{Recurrence}\label{recurrence}
In order to formulate our main result which gives the exact recurrence for the BLUEs under any rotation pattern we need to introduce two objects: a polynomial $Q_p$ and a matrix ${\bf S}$. They both look very technical and do not have immediate heuristic interpretations. Nevertheless they appear to be of essential importance for the final recurrence formula.
\subsection{Polynomial $Q_p$}
Recall that $T_k$, the $k$th Chebyshev polynomial of the first kind, is defined by
$$
T_k(x)=\cos(k\arccos\,x),\qquad k=0,1,\ldots.
$$

Define an $m\times m$ symmetric Toeplitz matrix polynomial function ${\bf T}_m$ by
\bel{Tm}
{\bf T}_m=\left[\begin{array}{cccccc}
T_0 & T_1 & T_2 & \cdots & T_{m-2} & T_{m-1} \\
T_1 & T_0 & T_1 & \cdots & T_{m-3} & T_{m-2} \\
\vdots & \vdots & \vdots & \ddots & \vdots & \vdots \\
T_{m-2} & T_{m-3} & T_{m-4} & \cdots & T_0 & T_1 \\
T_{m-1} & T_{m-2} & T_{m-3} & \cdots & T_1 & T_0 \end{array}\right]
\ee
and an $m\times m$ tridiagonal invertible matrix
\bel{Rm}
{\bf R}_m=\left[\begin{array}{cccccc}
1+\rho^2 & -\rho & 0 & \cdots & 0 & 0 \\
-\rho & 1+\rho^2 & -\rho & \cdots & 0 & 0 \\
0 & -\rho & 1+\rho^2 & \cdots & 0 & 0 \\
\vdots & \vdots & \vdots & \ddots & \vdots & \vdots \\
0 & 0 & 0 & \cdots & 1+\rho^2 & -\rho \\
0 & 0 & 0 & \cdots & -\rho & 1+\rho^2 \end{array}\right].
\ee
Note that ${\bf R}_m$ is non-singular.

For a cascade pattern $\ep$ with gaps sizes $m_1,\ldots,m_s$ and coverage $p$ define a polynomial $Q_p$ by
\bel{Qp}
Q_p(x)=(N-1)(1+\rho^2-2\rho x)+1-\rho^2-(1+\rho^2-2\rho x)^2\sum_{j=1}^s\,\mathrm{tr}({\bf T}_{m_j}(x){\bf R}_{m_j}^{-1}).
\ee

Since $\mathrm{tr}({\bf T}_m(x){\bf R}_m^{-1})$ is a polynomial of degree $m-1$ in $x$, $$\mathrm{deg}\,Q_p=2+\max_{1\le j\le s}\,(m_j-1)=p.$$

\subsection{Matrix ${\mathbf S}$}
Consider again a cascade pattern $\ep$ with coverage $p$ and $\#(H)=h=m_1+\ldots+m_s$. For complex numbers $d_1,\ldots,d_p$ define a $(ph+h+1)\times p(h+1)$ matrix ${\bf S}$ through its block structure
\bel{S}
{\bf S}={\bf S}(d_1,\ldots,d_p)=\left[\begin{array}{cccc}
\widetilde{\bf G}(d_1) & \widetilde{\bf G}(d_2) & \cdots & \widetilde{\bf G}(d_p) \\
{\bf G}(d_1) & {\bf 0} & \cdots & {\bf 0} \\
{\bf 0} & {\bf G}(d_2) & \cdots & {\bf 0} \\
\vdots & \vdots & \ddots & \vdots \\
{\bf 0} & {\bf 0} & \cdots & {\bf G}(d_p) \end{array}\right].
\ee
The blocks $\widetilde{\bf G}(d_i)$ are $(h+1)\times (h+1)$ matrices
\bel{tildeG}
\widetilde{\bf G}(d)=\tfrac{1}{1-\rho^2}\,\left[\begin{array}{cc} (N-1)(1-d\rho)+1-\rho^2 & (1-d\rho)\1_h^T \\
\; & \; \\
(1-d\rho)\1_h & \mathrm{diag}(\widetilde{\bf H}_{m_1},\ldots,\widetilde{\bf H}_{m_s})\end{array}\right]
\ee
with $\widetilde{\bf H}_m=\widetilde{\bf H}_m(d)$ being an $m\times m$ upper bi-diagonal matrix
\bel{tildeH}
\widetilde{\bf H}_m(d)=\left[\begin{array}{cccc} 1 & -d\rho & & \\
                                                   & \ddots      & \ddots &  \\
                                                 &        & \ddots & -d\rho \\
                                                 &              &   &  1 \end{array}\right].
\ee
The blocks ${\bf G}(d_i)$ are $h\times (h+1)$ matrices
\bel{bfG}
{\bf G}(d)= \tfrac{1}{1-\rho^2}\,[(1-d\rho)(d-\rho)\1_h,\;d\,\mathrm{diag}({\bf H}_{m_1},\ldots,{\bf H}_{m_s})],
\ee
where ${\bf H}_m={\bf H}_m(d)$ is an $m\times m$ tri-diagonal matrix \small
\bel{Hm}
{\bf H}_m(d)=\left[\begin{array}{cccc} 1+\rho^2 & -d\rho    &        &   \\
                                      \hspace{-2mm} -\rho/d     & \hspace{-3mm}\ddots & \hspace{-3mm} \ddots  &   \\
                                                   &\hspace{-5mm} \ddots   & \hspace{-3mm}\ddots & \hspace{-2mm} -d\rho \\
                                            &     & \hspace{-10mm} -\rho/d   & \hspace{-1mm}  1+\rho^2 \end{array}\right].
\ee

The numbers $d_1,\ldots,d_p$ considered above are related to (potentially complex) roots $x_1,\ldots,x_p$ of the polynomial $Q_p$ through the relation $2x_i=d_i+1/d_i$, and $|d_i|<1$, $i=1,\ldots,p$. Some more details are given in the remark below.

\begin{remark}\label{quadr}
Let $x\in \C$ be such that either $\Im\,x\ne 0$ or $\Re\,x\not\in[-1,1]$.

Then the equation $$\tfrac{1}{2}\left(d+\tfrac{1}{d}\right)=x$$ in $d$ has exactly two roots, say, $d_+(x)$ and $d_-(x)$ such that
$$|d_-(x)|<1\qquad\qquad\mbox{and}\qquad\qquad |d_+(x)|>1.$$

If additionally $\Im\,x=0$ then $d_+(x)$ and $d_-(x)$ are real.

By $x^*$ denote complex conjugate of $x$ with $\Im\,x\ne 0$. Then
$$d_-(x)=\left(d_-(x^*)\right)^*\qquad\mbox{and}\qquad d_+(x)=\left(d_+(x^*)\right)^*.$$
\end{remark}

\subsection{Main result}

Our main result gives the recursion of depth equal to the coverage $p$ of the cascade scheme together with analytic forms of the coefficients which are ready for numerical implementation. Actual examples of such implementations are presented in Section \ref{eg}. The proof we offer (see Appendix) is based on two basic assumptions concerning the polynomial $Q_p$ and the matrix ${\bf S}$.

\vspace{3mm}
ASSUMPTION I: The polynomial $Q_p$ has distinct roots $x_1,\ldots,x_p\not\in[-1,1]$.

\vspace{3mm}
ASSUMPTION II: The matrix ${\bf S}={\bf S}(d_1,\ldots,d_p)$, where $d_i=d_-(x_i)$, $i=1,\ldots,p$, is of full rank.

\begin{theorem}\label{main}
If {\em ASSUMPTION}s {\em I} and {\em II} are satisfied then for any $t\in \Z$ the recursion
\bel{rec}
\hat{\mu}_t=\sum_{k=1}^p\,a_k\hat{\mu}_{t-k}+\sum_{k=0}^p\,\r_k^T\,\X_{t-k}
\ee
holds with
\bel{ak}
a_k=(-1)^{k+1}\,\sum_{1\le j_1<\ldots<j_k\le p}\,d_{j_1}\ldots d_{j_k},\quad k=1,\ldots,p,
\ee
and
$$
\r_i=\sum_{m=1}^p\,\left[\left(v_i(d_m){\bf I}-v_{i-1}(d_m){\bf C}^T\right)\bD{\bf N}(d_m)\,\sum_{j\in H'}\,c_{j,m}\e_j\right],\quad i=0,1,\ldots,p,
$$
where $\e_0=\1$,  $H'=\{0\}\cup H$, $v_0(d)=1$, $v_{-1}(d)=0$,
\bel{vid}
v_i(d)=d^i-\sum_{l=1}^i\,a_ld^{i-l},\quad i=1,\ldots,p,
\ee
$\bD=({\bf I}-{\bf C}{\bf C}^T)^{-1}$, ${\bf N}(d)={\bf I}-d{\bf C}$  and with
$$\c=[(c_{j,1},\,j\in H'),\;(c_{j,2},\, j\in H'),\;\ldots,\;(c_{j,p},\,j\in H')]^T$$
being the unique solution (it exists due to {\em ASSUMPTION II}) of the linear system
$$
{\bf S}\c=(1,0,\ldots,0)^T\in\R^{ph+h+1}
$$

Moreover,
\bel{waria}
\var(\hat{\mu}_t)=\sum_{m=1}^p\,c_{0,m}.
\ee
\end{theorem}

In the next section we show how the above theoretical result can be applied in several basic schemes, in particular, in those which are used in real life surveys, while the proof of Theorem \ref{main} is given in the second part \ref{App2} of Appendix. It is based on a purely algebraic operator approach which is introduced earlier in the first part \ref{App1} of Appendix.

We would like to stress that intensive numerical experiments suggest that ASSUMPTIONS I and II may be universally satisfied, however at this moment we do not have mathematical proof of this fact (except the case $p=1,2$ and $p=3$ for a special rotation pattern). Thus applications of the above recursion formula (for $p>2$) in surveys have to be preceded by a numerical check (which is rather straightforward) that ASSUMPTIONS I and II are satisfied. Examples are given in Section \ref{eg}.

\section{Examples}\label{eg}

\subsection{Patterson's scheme, $p=1$.}
The cascade Patterson scheme is used e.g. for conducting the Labour
Force Survey in Australia ($N=n=8$, see Australian Bureau of Statistics (2002)) and Canada
($N=n=6$, see Singh et al. (1990)). There are no zeros in the pattern, hence $h=0$ and the
polynomial $Q_p=Q_1$, see \eqref{Qp}, does not contain the summand with the trace, that is
$$
    Q_1(x) = (N-1)(1+\rho^2-2\rho x) + 1-\rho^2.
$$
Its only root $x_1 = - \frac{1+\rho^2}{2\rho} -
\frac{1-\rho^2}{2(N-1)\rho}$ is real and satisfies $|x_1| >
\frac{1+\rho^2}{2|\rho|} > 1$, that is ASSUMPTION I is satisfied. It yields also real $d_1=d_-(x_1)$ of the form
$$
d_1 = \frac{N + (N-2)\rho^2-\sqrt{[N + (N-2)\rho^2]^2 - 4(N-1)^2\rho^2}}{2(N-1)\rho}.
$$

Moreover, $\mathbf{S}$ as defined in \eqref{S} is a $1\times 1$ matrix of the form $\mathbf{S} = \left[ (N-1)\frac{1-d_1\rho}{1-\rho^2} + 1  \right] \neq \mathbf{0}$, that is ASSUMPTION II trivially holds. Thus from Theorem \ref{main}, for all $t \in \Z$ we have
$$
\hat{\mu}_t = a_1 \hat{\mu}_{t-1} +\r_0^T\, \X_{t}  +\r_1^T\, \X_{t-1},
$$
where
$$
    \left\{\begin{array}{l}
    a_1 = d_1 \vspace{5pt} \\
    \r_0 = c_{0,1}\,{\bf N}(d_1)\,\1, \vspace{5pt} \\
    \r_1 = -c_{0,1}{\bf C}^T\,{\bf N}(d_1)\,\1,
    \end{array}\right.,
$$
where $$c_{0,1}=\tfrac{1}{(N-1)\tfrac{1-d_1\rho}{1-\rho^2}+1}.$$

Taking for example $N=6$ and $\rho=0.9$, we obtain for all $t$:
$$
    \hat{\mu}_t = 0.7942\ \hat{\mu}_{t-1}
    + \left[\begin{array}{r}
    \ 0.1765 \\
    \ 0.1765 \\
    \ 0.1765 \\
    \ 0.1765 \\
    \ 0.1765 \\
    \ 0.1176
    \end{array}
    \right]^T \X_{t}
    + \left[\begin{array}{r}
    \ 0.0000 \\
    -0.1588 \\
    -0.1588 \\
    -0.1588 \\
    -0.1588 \\
    -0.1588
    \end{array}
    \right]^T \X_{t-1}.
$$

\begin{remark}
Patterson (1950) considered the same scheme in the "classical"
model. The recurrence coefficient $a_1(t)$ was formally proved to
converge with $t\rightarrow\infty$ and the limit was shown to be $a_1$ as
given above. The vectors $\r_0(t)$ and
$\r_1(t)$, being continuous functions of $a_1(t)$,
converge to $\r_0$ and $\r_1$, respectively.
That is, the "stationary" solution is indeed consistent with
asymptotics of the "classical" one.
\end{remark}

\subsection{Schemes with gaps of size 1, $p=2$.}
The polynomial $Q_p=Q_2$, see \eqref{Qp}, has the following
form:
$$
Q_2(x) = -\tfrac{4h\rho^2}{1+\rho^2} x^2 - 2(N-2h-1)\rho x +
    (N-h-1)(1+\rho^2) + 1-\rho^2.
$$
As $1-\rho^2>0$, it is immediate that its discriminant
\bel{del}\Delta=4(N-2h-1)^2\rho^2+4\frac{4h\rho^2}{1+\rho^2}[(N-h-1)(1+\rho^2)+1-\rho^2]> 4\rho^2(N-1)^2>0.\ee
Thus $Q_2$ has two single real roots
$$
x_{\pm} = (1+\rho^2)\frac{-2(N-2h-1)\rho \pm \sqrt{\Delta}}{8h\rho^2}.
$$
Note that since the size of all gaps is one, then necessarily $N-h-1\ge h\ge 1$. Using this fact and inequality \eqref{del} we obtain
Therefore $$|x_{\pm}| > (1+\rho^2)\,\frac{N-h-1}{2|\rho|} \ge \tfrac{1+\rho^2}{2|\rho|}>1,\qquad \mbox{since}\qquad |\rho|\in(0,1).$$ Thus the ASSUMPTION I of Theorem \ref{main} is satisfied.

By Remark \ref{quadr} it follows that $d_1=d_{-}(x_-)=x_-+\sqrt{x_-^2-1}<0$ and $d_2=d_{-}(x_+)=x_+-\sqrt{x_+^2-1}>0$ are real numbers.

Since in this case $s=h$ and $m_1=\ldots=m_h=1$ we have $\widetilde{\bf H}_1(d_i)=1$ and ${\bf H}_1(d_i)=1+\rho^2$, $i=1,2$. Therefore the equation ${\bf S}\c=\e$ implies
$$
(1-d_i\rho)(d_i-\rho)c_{0,i}+(1+\rho^2)c_{k,i}=0,\qquad k=1,\ldots,h,\;i=1,2.
$$
Thus $c_{1,1}=c_{2,1}=\ldots=c_{h,1}$ and $c_{1,2}=c_{2,2}=\ldots=c_{h,2}$. Consequently, the system ${\bf S}\c=\e$ reduces to the system with four unknowns $c_{0,1}$, $c_{1,1}$, $c_{0,2}$ and $c_{1,2}$:
$$
\widetilde{\bf S}(c_{0,1},\,c_{1,1},\,c_{0,2},\,c_{1,2})^T=(1,\,0,\,0,\,0)^T
$$
with \small
$$
    \widetilde{\mathbf{S}} = \tfrac{1}{1-\rho^2}\,\left[\begin{array}{c c c c}
    (N-1)(1-d_1\rho) + 1-\rho^2 & h(1-d_1\rho) &
    (N-1)(1-d_2\rho) + 1-\rho^2 & h(1-d_2\rho) \\
    1-d_1\rho & 1 &
    1-d_2\rho & 1 \\
    (1-d_1\rho)(d_1-\rho) & d_1(1+\rho^2) &
    0 & 0 \\
    0 & 0 & (1-d_2\rho)(d_2-\rho) & d_2(1+\rho^2)
    \end{array}\right].
$$

\normalsize To show that $\widetilde{\mathbf{S}}$ is non-singular we first show that
\bel{rod}\rho(d_1+d_2) \ge 0.\ee To this end we first note that
\bel{roro}\rho(x_-+x_+)=-(1+\rho^2)\tfrac{N-2h-1}{2h}\le 0.\ee
Moreover,
$$
\rho(d_1+d_2)=\rho(x_-+x_+\sqrt{x_-^2-1}-\sqrt{x_+^2-1})=\rho(x_-+x_+)\left(1+\tfrac{x_--x_+}{\sqrt{x_-^2-1}+\sqrt{x_+^2-1}}\right)$$$$=
\tfrac{\rho(x_-+x_+)}{\sqrt{x_-^2-1}+\sqrt{x_+^2-1}}(\sqrt{x_-^2-1}+x_-+\sqrt{x_+^2-1}-x_+).
$$
Due to \eqref{roro} the last expression is non-negative since the second factor is strictly negative.
Now we are ready to consider the determinant
$$
    \det \widetilde{\mathbf{S}}= \frac{(d_2-d_1)\rho}{(1-\rho^2)^4}\,s(d_1,d_2),
$$
where
$$
s(d_1,d_2)=(1+\rho^2)[(N-1)(1-d_1\rho)(1-d_2\rho)+(1-\rho^2)(1+d_1d_2\rho^2)]$$$$+h(1-d_1\rho)(1-d_2\rho)(-1+(d_1+d_2)\rho+d_1d_2\rho^2-2\rho^2).
$$
We note that $|d_i|<1$, $i=1,2$, and thus $|d_1d_2|<1$. Consequently, we have $1+\rho^2>(1-d_1\rho)(1-d_2\rho)>0$, $1+d_1d_2\rho^2>0$. These inequalities together with \eqref{rod} yield
$$
s(d_1,d_2)>(1-d_1\rho)(1-d_2\rho)\left\{(N-1)(1+\rho^2)-h[1+d_1d_2\rho^2+2\rho^2]\right\}\,$$$$>(1-d_1\rho)(1-d_2\rho)[(N-h-1)(1+\rho^2)-2h\rho^2]
>(1-d_1\rho)(1-d_2\rho)(N-2h-1)(1+\rho^2)\ge 0.
$$
Consequently, $\det \widetilde{\mathbf{S}}\ne 0$.

Since $\mathrm{rank}\,{\bf S}=\mathrm{rank}\,\widetilde{\bf S}+2(h-1)$ we obtain $\mathrm{rank}\,{\bf S}=2(h+1)$
and thus the ASSUMPTION II of Theorem \ref{main} is satisfied. Moreover, $\widetilde{\mathbf{S}}^{-1}$ exists. Therefore
$$
(c_{0,1},\,c_{1,1},\,c_{0,2},\,c_{1,2})=(1,\,0,\,0,\,0)\,\left[\widetilde{\bf S}^{-1}\right]^T.
$$

Finally, we conclude that the recurrence has the following form:
$$
\hat{\mu}_t = a_1 \hat{\mu}_{t-1} +a_2 \hat{\mu}_{t-2} +\r_0^T\, \X_{t}+\r_1^T\, \X_{t-1} +\r_2^T\, \X_{t-2},
$$
where
$$
    \left\{\begin{array}{l}
    a_1 = d_1 + d_2 \vspace{5pt} \\
    a_2 = - d_1 d_2 \vspace{5pt} \\
    \r_0 = {\bf N}(d_1)\left[(c_{0,1}+c_{1,1})\1 -c_{1,1}\ep\,\right]+{\bf N}(d_2)\left[(c_{0,2}+c_{1,2})\1 -c_{1,2}\ep\,\right]
     \vspace{5pt} \\
    \r_1 = -(d_2{\bf I}+{\bf C}^T)\,{\bf N}(d_1)\left[(c_{0,1}+c_{1,1})\1 -c_{1,1}\ep\right]-(d_1{\bf I}+{\bf C}^T)\,{\bf N}(d_2)\left[(c_{0,2}+c_{1,2})\1 -c_{1,2}\ep\right]
     \vspace{5pt} \\
    \r_2 = d_2{\bf C}^T\,{\bf N}(d_1)\left[(c_{0,1}+c_{1,1})\1 -c_{1,1}\ep\,\right]+d_1{\bf C}^T\,{\bf N}(d_2)\left[(c_{0,2}+c_{1,2})\1 -c_{1,2}\ep\,\right]
    \end{array}\right.
$$

For example, let $N=7$, $h=2$, $H=\{3,6\}$ and let $\rho=0.5$. Then
$$
    Q_2(x) = -1.6 x^2 - 2x + 5.75
$$
and
$$
    \left\{\begin{array}{l}
    x_1 = -2.6211 \\
    x_2 = \hspace{6pt} 1.3711
    \end{array}\right.
    \quad \Rightarrow \quad
    {\setlength{\arraycolsep}{2pt}
    \left\{\begin{array}{r c l}
    d_+(x_1) & = & -5.0439 \\
    d_1=d_-(x_1) & = & -0.1983 \vspace{5pt} \\
    d_+(x_2) & = & \hspace{6pt} 2.3091 \\
    d_2=d_-(x_2) & = & \hspace{6pt} 0.4331
    \end{array}\right.}
    \quad \Rightarrow \quad
    \left\{\begin{array}{l}
    a_1 = 0.2348 \\
    a_2 = 0.0859
    \end{array}\right.
$$

Finally, \eqref{rec} assumes the form
$$
    {\setlength{\arraycolsep}{2pt}
    \begin{array}{r c l}
    \hat{\mu}_t & = & 0.2348\ \hat{\mu}_{t-1} + 0.0859\ \hat{\mu}_{t-2}
    + \vspace{5pt}\\
    & + & \left[\begin{array}{r}
    0.2171 \\
    0.1904 \\
    0.0000 \\
    0.2171 \\
    0.1904 \\
    0.0000 \\
    0.1850
    \end{array}
    \right]^T \hspace{-5pt}\X_{t} + \left[\begin{array}{r}
    -0.0093 \\
    -0.1086 \\
    0.0000 \\
    -0.0093 \\
    -0.1086 \\
    0.0000 \\
    0.0010
    \end{array}
    \right]^T \hspace{-5pt}\X_{t-1} + \left[\begin{array}{r}
    0.0000 \\
    0.0047 \\
    0.0000 \\
    -0.0476 \\
    0.0047 \\
    0.0000 \\
    -0.0476
    \end{array}
    \right]^T \hspace{-5pt}\X_{t-2}
    \end{array}}
$$

\subsection{Szarkowski's scheme, $p=3$.}
If there are $h_2$ gaps of size 2 and $h_1$ gaps of size 1 in the cascade pattern  the polynomial $Q_p=Q_3$, see \eqref{Qp}, assumes the form
$$
Q_3(x)=(N-1)(1+\rho^2-2\rho x) + 1-\rho^2 - (1+\rho^2-2\rho x)^2 \left(h_2 \frac{2\rho x+ 2(1+\rho^2)}{1+\rho^2+\rho^4}
    + h_1 \frac{1}{1+\rho^2} \right).
$$
The Szarkowski's scheme is defined by the cascade pattern $\ep=(1,1,0,0,1,1)^T$ (often denoted also as $2-2-2$), used e.g. by the Central
Statistical Office of Poland for conducting the Labour Force Survey
(known under the label BAEL), see Szarkowski and Witkowski (1994) or Popi\'nski (2006). Actually, such scheme is used also in LFS in other countries in Europe as well. Here $N=6$ and $H=\{3,4\}$. Thus $h_2=1$,  $h_1=0$ and
\bel{sz}
Q_3(x) = 5(1+\rho^2-2\rho x) + 1-\rho^2 -
        2(1+\rho^2-2\rho x)^2 \frac{\rho x+1+\rho^2}{1+\rho^2+\rho^4}.
\ee

Weso\l owski (2010)  proved that in this case $Q_3$ is either strictly increasing or decreasing in the whole domain and has
two complex conjugate roots $x_1$, $x_2$, and one real root $x_3\not\in[-1,1]$, meaning that the ASSUMPTION I of Theorem \ref{main} holds. It was also shown in that paper that the
matrix $\mathbf{S}$, in this case of dimensions $9\times 9$,
is invertible (meaning that the ASSUMPTION II of Theorem \ref{main} holds). Thus, just as for $p=1,2$, the recurrence \eqref{rec} for
Szarkowski's scheme always holds.

In general, even in the case $p=3$, verification of  ASSUMPTIONs I and II
of Theorem \ref{main} has to be done numerically, i.e.
after assigning the value to the correlation coefficient $\rho$. However,
it is worth noting that all performed simulations confirm existence
of the solution. Asymptotic approximation of the "classical" model
parameters was also observed in numerical experiments we performed.

The coefficients $a_1,a_2,a_3$ depend on $d_1=d_-(x_1)$, $d_2=d_-(x_2)=d_1^*$ and $d_3=d_-(x_3)$ in the
following way (see \eqref{ak}):
$$
    \left\{ \begin{array}{l}
    a_1 = d_1 + d_2 + d_3 \\
    a_2 = -(d_1d_2 + d_2d_3 + d_1d_3) \\
    a_3 = d_1 d_2 d_3
    \end{array} \right..
$$
For the Szarkowski scheme, taking for instance $\rho=0.7$ in \eqref{sz}, we obtain
$$
    \left\{\begin{array}{l}
    x_1 = -0.5668 -1.4069 \imath \\
    x_2 = -0.5668 +1.4069 \imath \\
    x_3 = \hspace{6pt} 1.1336
    \end{array}\right.
    \ \ \Rightarrow \ \
    {\setlength{\arraycolsep}{2pt}
    \left\{\begin{array}{r c l}
    d_+(x_1) & = & -1.0368 -3.1035 \imath \\
    d_1=d_-(x_1) & = & -0.0968 +0.2899 \imath \vspace{3pt} \\
    d_+(x_2) & = & -1.0368 +3.1035 \imath \\
    d_2=d_-(x_2) & = & -0.0968 -0.2899 \imath \vspace{3pt} \\
    d_+(x_3) & = & \hspace{6pt} 1.6675\\
    d_3=d_-(x_3) & = & \hspace{6pt} 0.5997
    \end{array}\right.}
    \ \ \Rightarrow \ \
    \left\{\begin{array}{l}
    a_1 = 0.4060 \\
    a_2 = 0.0227 \\
    a_3 = 0.0560
    \end{array}\right..
$$
Due to Theorem \ref{main} we get the following form of \eqref{rec}:
$$
    {\setlength{\arraycolsep}{2pt}
    \begin{array}{r c l}
    \hat{\mu}_t & = &
    0.4060\ \hat{\mu}_{t-1} + 0.0227\ \hat{\mu}_{t-2} + 0.0560\ \hat{\mu}_{t-3}
    + \vspace{5pt}\\
    & + & \left[\begin{array}{r}
    0.2862 \\
    0.2217 \\
    0.0000 \\
    0.0000 \\
    0.2862 \\
    0.2059
    \end{array}
    \right]^T \hspace{-5pt}\X_{t} + \left[\begin{array}{r}
    -0.0036 \\
    -0.2004 \\
    0.0000 \\
    0.0000 \\
    -0.0036 \\
    -0.1984
    \end{array}
    \right]^T \hspace{-5pt}\X_{t-1} + \left[\begin{array}{r}
    -0.0143 \\
    0.0026 \\
    0.0000 \\
    0.0000 \\
    -0.0143 \\
    0.0033
    \end{array}
    \right]^T \hspace{-5pt}\X_{t-2} + \left[\begin{array}{r}
    0.0000 \\
    0.0100 \\
    0.0000 \\
    0.0000 \\
    -0.0760 \\
    0.0100
    \end{array}
    \right]^T \hspace{-5pt}\X_{t-3}
    \end{array}}
$$

\subsection{CPS scheme, $p=9$}
Let us consider the well-known and widely studied 4-8-4 scheme, that is the cascade pattern is $$\ep=(1,1,1,1,0,0,0,0,0,0,0,0,1,1,1,1)^T$$ which is used in the US in the Current Population Survey, see U.S. Bureau of Census (2002). In this case $N=16$, $h=8$ and
$H=\{5,\ldots,12\}$.  We do not have any analytical proof that ASSUMPTIONs I and II are satisfied in this scheme for any $\rho$.

The polynomial $Q_p=Q_9$, see \eqref{Qp}, is of degree 9 and has the form
$$
Q_9(x)=15(1+\rho^2-2\rho x)+1-\rho^2-(1+\rho^2-2\rho x)^2\,\mathrm{tr}\left({\bf T}_8(x)\,{\bf R}_8^{-1}\right).
$$

Consequently, its analysis, as well as analysis of matrix ${\bf S}$ (which is of dimension $81\times 81$ in this scheme), can be done numerically, after assigning some value for $\rho$. To make use of the result of Theorem \ref{main} we need to check numerically that ASSUMPTIONs I and II are satisfied for a given concrete value for $\rho$. We checked that they hold true for several values for $\rho$ picked up at random from the interval $(-1,1)$.

Taking for instance $\rho=0.9$, we obtain that $Q_9$ has 8 complex roots and 1 real root of the form
$$
    \left\{\begin{array}{l}
    x_1 = -0.7667 - 0.0208 \imath \\
    x_2 = -0.7667 + 0.0208 \imath \\
    x_3 = -0.1746 - 0.0320 \imath \\
    x_4 = -0.1746 + 0.0320 \imath \\
    x_5 = \hspace{6pt} 0.4989 - 0.0284 \imath \\
    x_6 = \hspace{6pt} 0.4989 + 0.0284 \imath \\
    x_7 = \hspace{6pt} 0.9391 - 0.0121 \imath \\
    x_8 = \hspace{6pt} 0.9391 + 0.0121 \imath \\
    x_9 = -1.0006
    \end{array}\right.
    \ \ \Rightarrow \ \
    {\setlength{\arraycolsep}{2pt}
    \left\{\begin{array}{r c l}
    d_1=d_-(x_1) & = & -0.7419 - 0.6220 \imath \\
    d_2=d_-(x_2) & = & -0.7419 + 0.6220 \imath \\
    d_3=d_-(x_3) & = & -0.1689 - 0.9532 \imath \\
    d_4=d_-(x_4) & = & -0.1689 + 0.9532 \imath \\
    d_5=d_-(x_5) & = & \hspace{6pt} 0.4825 - 0.8389 \imath \\
    d_6=d_-(x_6) & = & \hspace{6pt} 0.4825 + 0.8389 \imath \\
    d_7=d_-(x_7) & = & \hspace{6pt} 0.9064 - 0.3335 \imath \\
    d_8=d_-(x_8) & = & \hspace{6pt} 0.9064 + 0.3335 \imath \\
    d_9=d_-(x_9) & = & -0.9682
    \end{array}\right.}
    \ \ \Rightarrow \ \
    \left\{\begin{array}{l}
    a_1 = 0.7429 \\
    a_2 = 0.0019 \\
    a_3 = 0.0023 \\
    a_4 = 0.0029 \\
    a_5 = 0.0037 \\
    a_6 = 0.0049 \\
    a_7 = 0.0066 \\
    a_8 = 0.0088 \\
    a_9 = 0.0119
    \end{array}\right..
$$

The coefficient $a_1$ is dominant in terms of absolute value. The
second largest, $a_9$ is smaller by one order of magnitude and the
other coefficients by at least two. Results for other values of the
parameter $\rho$ behave similarly.

\section{Discussion} The main result of the paper is an explicit recurrence formula for the best linear unbiased estimator (BLUE) of the mean on any occasion in repeated surveys with any cascade rotation pattern. The principal novelty lies in allowing for gaps in the pattern. The results which have been known earlier either dealt with patterns with no gaps or with estimators which were not BLUEs. The approach, we developed, is heavily based on algebra of matrices and linear operators of infinite dimension as well as on properties of Chebyshev polynomials. Unfortunately, the explicit recursive formula we obtained in Theorem \ref{main} needs two, seemingly technical, assumptions: ASSUMPTION I on localization of roots of a polynomial $Q_p$ and ASSUMPTION II on rank of matrix ${\bf S}$. It is worth to emphasize that both these objects, $Q_p$ and ${\bf S}$, depend ONLY on two parameters; the rotation pattern $\ep$ and  the correlation coefficient $\rho$. It is known that these two assumptions are satisfied if the coverage of the pattern $p=1$ or $p=2$ for any cascade scheme and  $p=3$ for 2-2-2 scheme. It is not known if they are satisfied in general. However numerical experiments allow to formulate a conjecture that this is really the case. In these experiments we considered many different rotation patterns. For each such a pattern we considered several values for $\rho\in(-1,1)$. Having the rotation pattern $\ep$ and the value of $\rho$ chosen, we built respective polynomial $Q_p$ and matrix ${\bf S}$. Numerically we looked for roots of $Q_p$. Often these roots were complex, but when they were real they were located outside of the interval $(-1,1)$ in all the experiments (that is, ASSUMPTION I was satisfied). Then we tried to solve numerically the equation ${\bf S}\c = (1,0,\ldots,0)\in\R^{ph+h+1}$. Again, in all the experiments we obtained the unique solution, meaning that ${\bf S}$ was of full rank (that is, ASSUMPTION II was also satisfied). We do believe that both the assumptions are always satisfied but a mathematical proof of both these facts is probably hard, though a paper with the proof that ASSUMPTION I is satisfied for any cascade pattern  with a single gap of any size and any $\rho\in(-1,1)$ is under preparation.

There is other type of limitations of the method we propose - they are due to the model constraints. In particular, in the model the correlations are exponential (as in the original Patterson model). This property is very important for the argument we use, e.g. it makes the covariance matrix ${\bf C}$ nilpotent of degree $N$, that is $N$ is the smallest value of $j$ such that ${\bf C}^j=0$. Moreover, it has been observed (see Example 4.5 in Kowalski, 2009) that other covariance models may lead to major difficulties in analysis of the formula for the variance of the estimators. There is a possibility that some reasonable departures from the exponential correlation assumption, as e.g. $\cov(X_{i,j},X_{k,l})=\theta+(1-\theta)\rho^{|j-l|}\delta_{i,k}$ for a $\theta\in[0,1]$ (see Lent et al. (1999), in particular their Table 1, its discussion as well as additional references) can lead to treatable formulas for the variance. Such a covariance model is probably the first one to look at in any future research aiming at extension of the model.

In the model we also assumed that expectations on a given occasion are all the same and depend only on the occasion number: $\E\,X_{i,j}=\mu_j$. However other models containing may be of interest, e.g. $\E\,X_{i,j}=\mu_j+a_i$ (see Bailar, 1975). Here the adjustments $a_i$ can be understood as time-in-sample-bias caused by the number of occasions in which unit $i$ participated in the survey. Of course, if $a_i$ is known, there is no problem: just adjust $X_{i,j}$ by subtracting $a_i$ and use the approach we developed. If it is not known, the operational (but not mathematical) solution would be to adjust $X_{i,j}$'s with suitable estimators of $a_i$'s (obtained outside the model we analyze). The exact mathematical solution is not known and is worth to pursue.

Another aspect, which is of interest within the model considered in this paper, is the question of recurrence for the BLUE of a change of the mean $\mu_t-\mu_{t-1}$. We do believe that this question can be approached through the methods developed in this paper. Nevertheless, we expect it will need a lot of work in careful adaptations of the algebraic techniques used  above.

It is worth also to mention that the model considered in the paper has an infinite time horizon, why there is always finite number of occasions in real surveys. As already mentioned in Introduction, the results we obtained seem to be reasonable approximation of the finite horizon case, when coefficient of recursion \eqref{eqn_rekur} depend on $t$. In particular, numerical experiments, performed for a wide range of $\rho\in(-1,1)$ and several different cascade patterns $\eps$, show that e.g. the value of the coefficients $a_i^{(t)}$ (for the finite horizon) was roughly the same as $a_i$ (for the infinite horizon) already for $t\approx 10$. The same behavior was observed for the variances of the estimators. Nevertheless, the convergence has been mathematically established only in the case $p=1$. Analytical bounds for the speed of convergence at present seem also to be out of reach.

It is interesting to know how the estimators, obtained here, work in real surveys. Such question needs access to real data and gaining some interest of practitioners in the theoretical solutions we proposed. Very likely the exact formulas given in Theorem \ref{main} may need some adjustments due to the discussed limitations of the model.

\vspace{5mm}
\section{Appendix}
\subsection{Algebra of shift operators}\label{App1}
In the first part of Appendix we introduce and analyze an algebraic operator formalism which is crucial for the proof of our main result (given in Subsection \ref{App2}).

For a sequence of vectors $\overline{\x}=(\x_0,\x_1,\x_2,\ldots)$, $\x_i\in\R^N$, define shifts to the left and to the right by
$$\Le(\overline{\x})=(\x_1,\x_2,\x_3,\ldots)\qquad \mbox{left shift},$$
$$\Ri(\overline{\x})=(\0,\x_0,\x_1,\ldots)\qquad \mbox{right shift}.$$
Note that $\Le\,\Ri=\Id$ (identity), but
\bel{irl}
(\Id-\Ri\,\Le)\,\overline{\x}=(\x_0,\,\0,\,\0,\,\ldots)=\x_0\overline{e},
\ee
where $\overline{e}=(1,\,0,\,0,\,\ldots)$.

For any $M\times N$ matrix ${\bf A}$ define
$$
{\bf A}\,\overline{\x}=({\bf A}\,\x_0,\,{\bf A}\,\x_1,\,{\bf A}\,\x_2,\,\ldots).
$$
In particular, for a complex (real) number $a$, taking ${\bf A}=a\,{\bf I}$ we have
$$
a\,\overline{\x}=(a\,\x_0,\,a\,\x_1,\,a\,\x_2,\,\ldots).
$$

Moreover, by the above definitions, for any $i,j\ge 0$
$$
\Ri^i\,\Le^j\,{\bf A}\,\overline{\x}={\bf A}\,\Ri^i\,\Le^j\,\overline{\x}.
$$

For a constant sequence of vectors $\overline{\x}=(\x,\x,\x,\ldots)$ we have $\Le\,\overline{\x}=\overline{\x}$
and thus for any $i,j\ge 0$
\bel{lirj}
\Le^i\Ri^j\,\overline{\x}=\left\{\begin{array}{ll}
\overline{\x}, & \mbox{for}\;i\ge j, \\
\; & \; \\
\Ri^{j-i}\,\overline{\x}, & \mbox{for}\;i<j.
\end{array}\right.
\ee

If $N=1$ we write $\overline{\y}=\overline{y}=(y_0,y_1,y_2\ldots)$, $y_i\in\R$, and $\mathrm{L}:=\Le$, $\mathrm{R}:=\Ri$.
Note that, for $\overline{y}=(y^n)_{n\ge 0}$ we have \bel{Ly}\mathrm{L}^j\,\overline{y}=y^j\,\overline{y}\ee
and thus
$$
\mathrm{L}^j\,\mathrm{R}^i\,\overline{y}=\left\{\begin{array}{ll}
y^{j-i}\,\overline{y}, & \mbox{for}\;j\ge i, \\
\; & \; \\
\mathrm{R}^{i-j}\,\overline{y}, & \mbox{for}\;j<i.
\end{array}\right.
$$

For any $\overline{y}=(y_n)_{n\ge 0}$ and any $\overline{\x}=(\x_n)_{n\ge 0}$ define $\overline{y}\,\overline{\x}=(y_n\,\x_n)_{n\ge 0}$.
Then for any complex (real) numbers $\alpha$, $\beta$, any $M\times N$ matrices $\bf A$, $\bf B$, any $i,j,k,m\ge 0$,
\begin{equation}\label{AB}
\left(\alpha\,{\bf A}\,\Ri^i\,\Le^j+\beta\,{\bf B}\,\Le^m\,\Ri^k\right)\,\overline{y}\,\overline{\x}
=\left(\alpha\,\mathrm{R}^i\,\mathrm{L}^j\,\overline{y}\right)\,\left({\bf A}\,\Ri^i\,\Le^j\,\overline{\x}\right)
+\left(\beta\,\mathrm{L}^m\,\mathrm{R}^k\,\overline{y}\right)\,\left({\bf B}\,\Le^m\,\Ri^k\,\overline{\x}\right).
\end{equation}

Note also that if $\overline{\x}=(\x,\,\x,\,\ldots)$ is a constant sequence, then
\bel{rilj}
\Ri^i\,\Le^j\,\overline{y}\,\overline{\x}=\left(\mathrm{R}^i\,\mathrm{L}^j\,\overline{y}\right)\,\overline{\x}\qquad\mbox{and}\qquad \Le^j\,\Ri^i\,\overline{y}\,\overline{\x}=\left(\mathrm{L}^j\,\mathrm{R}^i\,\overline{y}\right)\,\overline{\x}.
\ee

\begin{lemma}
Let $v_i$, $i=1,\ldots,p$, be functions defined in \eqref{vid}, where $a_1,\ldots,a_p$ are arbitrary numbers. Let $\overline{\x}=(\x,\,\x,\,\ldots)$ and $\overline{y}=(y^n)_{n\ge 0}$. Then
for any $i=1,\ldots,p$
\bel{zerid}
\Le^i\left(\Id-\sum_{j=1}^p\,a_j\Ri^j\right)=\left(\Le^p-\sum_{j=1}^p\,a_j\Le^{p-j}\right)\Ri^{p-i},
\ee
\bel{firstid}
(\Id-\Ri\Le)\left(\Le^p-\sum_{j=1}^p\,a_j\Le^{p-j}\right)\Ri^{p-i}\,\overline{y}\,\overline{\x}=v_i(y)(\x_0,\0,\0,\ldots)
\ee
and
\bel{secondid}
\left(\Le^p-\sum_{j=1}^p\,a_j\Le^{p-j}\right)\,\overline{y}\,\overline{\x}=v_p(y)\,\overline{y}\,\overline{\x}.
\ee
\end{lemma}

\begin{proof} First, we prove \eqref{secondid}. By \eqref{AB}
$$
\left(\Le^p-\sum_{j=1}^p\,a_j\Le^{p-j}\right)\,\overline{y}\,\overline{\x}=(\mathrm{L}^p\,\overline{y})\,\Le^p\overline{\x}
-\sum_{j=1}^p\,a_j(\mathrm{L}^{p-j}\overline{y})\,(\Le^{p-j}\overline{\x})
$$
Note that $\mathrm{L}^k\overline{y}=y^k\overline{y}$ and $\Le^k\overline{\x}=\overline{\x}$ for any $k=0,1,\ldots$. Therefore
$$\left(\Le^p-\sum_{j=1}^p\,a_j\Le^{p-j}\right)\,\overline{y}\,\overline{\x}=\left[\left(y^p-\sum_{m=1}^p\,a_m\,y^{p-m}\right)\overline{y}\right]\,\overline{\x}.
$$
Now \eqref{secondid} follows by the definition \eqref{vid} for $i=p$.

Again, from \eqref{lirj}, \eqref{AB} and \eqref{rilj} it follows that
$$
(\Id-\Ri\Le)\left(\Le^p-\sum_{j=1}^p\,a_j\Le^{p-j}\right)\Ri^{p-i}\,\overline{y}\,\overline{\x}=
\left[(\mathrm{I}-\mathrm{R}\mathrm{L})\,\left(\mathrm{L}^p-\sum_{j=1}^p\,a_j\mathrm{L}^{p-j}\right)\mathrm{R}^{p-i}\,\overline{y}\right]\,\overline{\x}.
$$
Since for any $k\in\{0,1,\ldots,p\}$
$$
\left(\mathrm{L}^p-\sum_{j=1}^p\,a_j\,\mathrm{L}^{p-j}\right)\,\mathrm{R}^{p-k}\,\overline{y}
=y^k\,\overline{y}-\sum_{j=1}^k\,a_j\,y^{k-j}\,\overline{y}-\sum_{j=k+1}^p\,a_j\,\mathrm{R}^{j-k}\,\overline{y}=v_k(y)\overline{y}-\sum_{j=k+1}^p\,a_j\,\mathrm{R}^{j-k}\,\overline{y}
$$
then
$$(\mathrm{I}-\mathrm{R}\,\mathrm{L})\,\left(\mathrm{L}^p-\sum_{j=1}^p\,a_j\,\mathrm{L}^{p-j}\right)\,\mathrm{R}^{p-k}\,\overline{y}
=v_k(y)\,\overline{e}
$$
and thus \eqref{firstid} follows.

The identity \eqref{zerid} follows by \eqref{lirj} since
$$
\Le^i\left(\Id-\sum_{j=1}^p\,a_j\Ri^j\right)=\Le^i-\sum_{j=1}^p\,a_j\,\Le^i\Ri^j=
\Le^p\Ri^{p-i}-\sum_{j=1}^p\,a_j\Le^{p-j}\Ri^{p-i}.
$$
\end{proof}

\begin{lemma}\label{delt}
Let $\D$ be an operator on the space of sequences of vectors from $\R^N$ defined by
\bel{Dop}
\D=\Id+\sum_{k=1}^{N-1}\,\left({\bf C}^k\,\Le^k+({\bf C}^T)^k\,\Ri^k\right),
\ee
where ${\bf C}$ is the covariance matrix defined in Section \ref{model}.

The operator $\D$ is invertible and
\bel{invD}
\D^{-1}=(\Id-{\bf C}^T\Ri)\bD(\Id-{\bf C}\Le).
\ee
\end{lemma}
\begin{proof}
Note that ${\bf I}-{\bf C}\,{\bf C}^T=\mathrm{diag}(1-\rho^2,\ldots,1-\rho^2,1)$. Consequently, $\bD=\left({\bf I}-{\bf C}\,{\bf C}^T\right)^{-1}$ is well defined.
Note also that $\sum_{k=0}^{N-1}\,{\bf C}^k\Le^k$ is invertible and its inverse is $\Id-{\bf C}\Le$. Similarly, $\sum_{k=0}^{N-1}\,({\bf C}^T)^k\Ri^k$ is invertible and its inverse is $\Id-{\bf C}^T\Ri$.

Therefore
$$[(\Id-{\bf C}^T\Ri)\bD(\Id-{\bf C}\Le)]^{-1}=(\Id-{\bf C}\Le)^{-1}\bD^{-1}(\Id-{\bf C}^T\Ri)^{-1}=
\left(\sum_{k=0}^{N-1}\,{\bf C}^k\Le^k\right)\,({\bf I}-{\bf C}{\bf C}^T)\,\left(\sum_{j=0}^{N-1}\,({\bf C}^T)^j\Ri^j\right)$$$$
=\sum_{k,j=0}^{N-1}\,{\bf C}^k({\bf C}^T)^j\,\Le^k\Ri^j\,-\,\sum_{k,j=1}^{N-1}\,{\bf C}^k({\bf C}^T)^j\,\Le^k\Ri^j
=\D+\sum_{k,j=1}^{N-1}\,{\bf C}^k({\bf C}^T)^j\,\Le^{k-1}(\Le\Ri-\Id)\Ri^{j-1}=\D.
$$
\end{proof}

\subsection{Proof of the recurrence}\label{App2}
\begin{proof}[Proof of Theorem \ref{main}]
Note first that since $d_1,\ldots,d_p$ are either real or come in conjugate pairs (see Remark \ref{quadr}) it follows from \eqref{ak} that $a_1,\ldots,a_p$ are real numbers.

Recall that $\e_0=\1$ and denote $\overline{\e}_j=(\e_j,\,\e_j,\,\ldots)$, $j\in H'=\{0\}\cup H$. Recall that the $N\times N$ diagonal matrix $\bD$ is defined as
$$\bD=({\bf I}-{\bf C}\,{\bf C}^T)^{-1}=\tfrac{1}{1-\rho^2}\mathrm{diag}(1,\,\ldots,\,1,\,1-\rho^2).$$

With $d_1,\ldots,d_p$ and $\c$ as defined in Theorem \ref{main} let (see \eqref{invD})
\bel{overw}
\overline{\w}=(\w_0,\,\w_1,\,\ldots)=\D^{-1}\,\sum_{m=1}^p\,\sum_{j\in H'}\,c_{j,m}\,\overline{d}_m\,\overline{\e}_j,
\ee
where $\overline{d}_m=(1,\,d_m,\,d_m^2,\,\ldots)$, $m=1,\ldots,p$. Note that $||\w_i||$ (the length of the vector $\w_i$) is of order $(\max_{1\le m\le p}\,|d_m|)^i$, $i=0,1,\ldots$. By Remark \ref{quadr} and ASSUMPTION II we have $\max_{1\le m\le p}\,|d_m|\in (0,1)$. Hence \eqref{estyma} is a correct definition of a random series (with bounded variance).

Consequently, it suffices to show that:
\begin{enumerate}[{\bf 1.}]
\item The sequence $\overline{\w}$ defined in \eqref{overw} is the sequence of optimal weights. To this end we note that the variance of any linear estimator $\sum_{i=0}^{\infty}\,\u_i^T\,\X_i$, $\u_i\in\R^N$, $i=0,1,\ldots$, has the form
     \bel{warian}
     \var\,\sum_{i=0}^{\infty}\,\u_i^T\,\X_i=\sum_{i=0}^{\infty}\,\u_i^T\u_i+2\sum_{i=0}^{\infty}\,\sum_{k=1}^{N-1}\,\u_i^T{\bf C}^k\u_{i+k}.
     \ee
     We need to show that $\overline{\u}=(\u_i)_{i\ge 0}:=\overline{\w}$ with $\overline{\w}$ as defined in \eqref{overw} minimize this expression under the constraints \eqref{unb} and \eqref{pat}. Since the above variance as a function of $\overline{\u}$ is convex then the problem has the unique solution.
     Using the standard Lagrange method, that is differentiating the Lagrange function (with multipliers $(\lambda_{j,i})_{j\in H',\,i\ge 0}$)
     $$
     V(\overline{\u})=\sum_{i=0}^{\infty}\,\u_i^T\u_i+2\sum_{i=0}^{\infty}\,\sum_{k=1}^{N-1}\,\u_i^T{\bf C}^k\u_{i+k}
     -2\sum_{i=0}^{\infty}\,\sum_{j\in H'}\,\lambda_{j,i}\u_i^t\e_j,
     $$
     with respect to $(\u_i)_{i\ge 0}$ and comparing the derivatives to zero, equivalently, we need to show that there exist real numbers (Lagrange multipliers) $\lambda_{j,l}$, $j\in H'$, $l=0,1,\ldots$, such that
    \bel{Dw}
    \D\,\overline{\w}=\left[\Id+\sum_{k=1}^{N-1}\,\left({\bf C}^k\,\Le^k+({\bf C}^T)^k\,\Ri^k\right)\right]\,\overline{\w}=\overline{\underline{\Lambda}},
    \ee
    where $\overline{\w}$ is defined in \eqref{overw} and $\overline{\underline{\Lambda}}=(\underline{\Lambda}_0,\,\underline{\Lambda}_1,\,\ldots)$ with
    $$
    \underline{\Lambda}_l=\sum_{j\in H'}\,\lambda_{j,l}\,\e_j,\qquad l=0,1,\ldots
    $$

\item The constraints \eqref{unb} and \eqref{pat} are satisfied for $\overline{\w}$ as defined in \eqref{overw}.

\item The basic recurrence \eqref{rec} holds true with $\overline{\w}$ defined in \eqref{overw}, that is the sequence $\overline{\r}$ defined by
\bel{ery}
\overline{\r}:=\left(\Id-\sum_{m=1}^p\,a_m\,\Ri^m\right)\overline{\w}
\ee
has to satisfy
\bel{idp}
\Le^{p+1}\,\overline{\r}=\overline{\0}
\ee
and for any $i=0,1,\ldots,p$
\bel{idrl}
(\Id-\Ri\,\Le)\Le^i\,\overline{\r}=\sum_{m=1}^p\,\left[\left(v_i(d_m){\bf I}-v_{i-1}(d_m){\bf C}^T\right){\bf N}(d_m)\,\sum_{j\in H'}\,c_{j,m}\e_j\right]\,\overline{e},
\ee
where ${\bf N}(d)=\bD({\bf I}-d{\bf C})$.
\end{enumerate}

{\bf Ad. 1.} We will show that \eqref{Dw} holds with
\bel{lamb}
\lambda_{j,l}=\sum_{m=1}^p\,c_{j,m}\,d_m^l,\qquad j\in H',\;l=0,1,\ldots
\ee

By definition \eqref{overw} of $\overline{\w}$ we have
$$
\D\,\overline{\w}=\sum_{m=1}^p\,\sum_{j\in H'}\,c_{j,m}\,\overline{d}_m\,\overline{\e}_j=
\left(\sum_{j\in H'}\,\sum_{m=1}^p\,c_{j,m} d_m^l\,\e_j,\,l=0,1,\ldots\right).$$
Therefore, by definition of $\lambda_{j,l}$'s we obtain
$$\D\,\overline{\w}
=\left(\sum_{j\in H'}\,\lambda_{j,l}\,\e_j\right)
=\left(\underline{\Lambda}_0,\,\underline{\Lambda}_1,\,\ldots\right)=\overline{\underline{\Lambda}}.
$$

To see that $\lambda_{j,l}$ as defined through \eqref{lamb} are real numbers take first conjugates of both sides of ${\bf S}\,\c=\e$. Note that
$$
{\bf S}^*={\bf S}^*(d_1,\ldots,d_p)={\bf S}(d^*_1,\ldots,d^*_p).
$$
Since $d_1,\ldots,d_p$ are either real or come in conjugate pairs (see Rem. \ref{quadr}) the equation ${\bf S}^*\,\c^*=\e$ implies that for for any $j\in H'$ and any $m=1,\ldots,p$  either $\Im\,d_m=0$ and then $c_{j,m}$ is real or  $\Im\,d_m\ne 0$ and then there exists $n\ne m$ (with $d^*_n=d_m$) such that $c^*_{j,n}=c_{j,m}$. Therefore the quantities $c_{j,m}\,d_m^l$ in \eqref{lamb} are either real or come in conjugate pairs. Consequently, by \eqref{lamb} it follows that $\lambda_{j,l}$ is real.

{\bf Ad. 2.}
Note that applying \eqref{irl} and \eqref{AB}  to \eqref{overw} after an easy algebra we get
$$
\w_0=\sum_{m=1}^p\,\sum_{j\in H'}\,c_{j,m}\,{\bf N}(d_m)\,\e_j
$$
and
$$
\w_i=\sum_{m=1}^p\,\sum_{j\in H'}\,c_{j,m}d_m^{i-1}\,(d_m{\bf I}-{\bf C}^T)\,{\bf N}(d_m)\,\e_j,\qquad i=1,2,\ldots
$$

Let us rewrite the constraints \eqref{unb} and \eqref{pat} using the above formulas for $\w_0$ and $\w_i$, $i\ge 1$. The constraint \eqref{unb} for $i=0$ with $\w_0$ as defined above takes on the form
\bel{eq1}
\sum_{m=1}^p\,\sum_{j\in H'}\,c_{j,m}\,\1^T\,{\bf N}(d_m)\e_j=1
\ee
and for $i\ge 1$
\bel{eq2}
\sum_{m=1}^p\,\sum_{j\in H'}\,c_{j,m}\,d_m^{i-1}\1^T(d_m{\bf I}-{\bf C}^T)\,{\bf N}(d_m)\e_j=0.
\ee

The constraint \eqref{pat} for $i=0$, that is for $\w_0$, has the form
\bel{eq3}
\sum_{m=1}^p\,\sum_{j\in H'}\,c_{j,m}\,\e_k^T\,{\bf N}(d_m)\e_j=0,\qquad k\in H.
\ee
for $i>0$ it has the form
\bel{eq4}
\sum_{m=1}^p\,\sum_{j\in H'}\,c_{j,m}\,d_m^{i-1}\e_k^T(d_m{\bf I}-{\bf C}^T)\,{\bf N}(d_m)\e_j=0,\qquad k\in H.
\ee

Note that $N\times N$ matrix
$$
{\bf N}(d)=\tfrac{1}{1-\rho^2}\,\left[\begin{array}{ccccc}
                                      1 & -\rho d  & \ddots & 0 & 0 \\
                                      0 & 1 &  \ddots & \ddots & 0 \\
                                      \ddots & \ddots & \ddots & \ddots & \ddots \\
                                      0 & \ddots & \ddots & 1 & -\rho d \\
                                      0 & 0 & \ddots & 0 & 1-\rho^2
                                      \end{array}\right]
                                      $$
and $(d{\bf I}-{\bf C}^T)\,{\bf N}(d)=\tfrac{d}{1-\rho^2}{\bf H}_N(d)$ - see \eqref{Hm}. Thus, by elementary computations, we get
\bel{111}
\e_k^T{\bf N}(d)\e_j=\tfrac{1}{1-\rho^2}\,\left\{\begin{array}{ll}
(N-1)(1-d\rho)+1-\rho^2, & k=j=0, \\
1-d\rho, & k=0,\,j\in H\;\mbox{or}\;k\in H,\,j=0, \\
\begin{array}{l} 1, \\ -d\rho, \\ 0, \end{array} &
\left.\begin{array}{l} k=j, \\ k=j-1, \\  \mbox{otherwise}, \end{array}\right\} k,j\in H\end{array}\right.
\ee
and
\bel{222}
\e_k^T(d{\bf I}-{\bf C}^T){\bf N}(d)\e_j=\tfrac{1}{1-\rho^2}\,\left\{\begin{array}{ll}
(N-1)(1-d\rho)(d-\rho)+d(1-\rho^2), & k=j=0, \\
(1-d\rho)(d-\rho), & k=0,\,j\in H\;\mbox{or}\;k\in H,\,j=0, \\
\begin{array}{l} -\rho, \\ d(1+\rho^2), \\ -d^2\rho, \\ 0, \end{array} &
\left.\begin{array}{l} k=j+1, \\ k=j, \\ k=j-1, \\  \mbox{otherwise}, \end{array}\right\} k,j\in H.\end{array}\right.
\ee

Due to \eqref{111} and \eqref{222}, the constraints \eqref{eq1}, \eqref{eq2}, \eqref{eq3} and \eqref{eq4} can be rewritten in a matrix form as
\bel{linrow}
\left[\begin{array}{cccc}
\widetilde{\bf G}(d_1) & \widetilde{\bf G}(d_2) & \cdots & \widetilde{\bf G}(d_p) \\
\bar{\bf G}(d_1) & \bar{\bf G}(d_2) & \cdots & \bar{\bf G}(d_p) \\
d_1\bar{\bf G}(d_1) & d_2\bar{\bf G}(d_2) & \cdots & d_p\bar{\bf G}(d_p) \\
\vdots & \vdots & \ddots & \vdots \\
d_1^i\bar{\bf G}(d_1) & d_2^i\bar{\bf G}(d_2) & \cdots & d_p^i\bar{\bf G}(d_p) \\
\vdots & \vdots & \ddots & \vdots
\end{array}\right]\,\c=\overline{e},
\ee
where $\widetilde{\bf G}(d)$ is defined through \eqref{tildeG} and \eqref{tildeH},
$$
\bar{\bf G}(d)=\tfrac{d}{1-\rho^2}\,\left[\begin{array}{cc} {\bf H}_{11}(d) & {\bf H}_{12}(d) \\
{\bf H}_{21}(d) & {\bf H}_{22}(d)\end{array}\right]
$$
with
$$
{\bf H}_{11}(d)=(N-1)(1-\rho d)(1-\rho/d)+1-\rho^2,\qquad {\bf H}_{12}={\bf H}_{21}^T=(1-\rho d)(1-\rho/d)\,\1_h^T,
$$
$$
{\bf H}_{22}(d)=\mathrm{diag}({\bf H}_1(d),\ldots,{\bf H}_s(d)),
$$
and matrices ${\bf H}_i(d)$, $i=1,\ldots,s$, are defined in \eqref{Hm}.

The infinite matrix at the left hand side of \eqref{linrow} can be  written as
$$
\left[\begin{array}{ccccc}
{\bf I} & {\bf 0} & {\bf 0} & \cdots & {\bf 0} \\
{\bf 0} & {\bf I} & {\bf I} & \cdots & {\bf I} \\
{\bf 0} & d_1{\bf I} & d_2{\bf I} & \cdots & d_p{\bf I} \\
\vdots & \vdots & \vdots & \ddots & \vdots \\
{\bf 0} & d_1^i{\bf I} & d_2^i{\bf I} & \cdots & d_p^i{\bf I} \\
\vdots & \vdots & \vdots & \ddots & \vdots
\end{array}\right]\;\left[\begin{array}{cccc} \widetilde{\bf G}(d_1) & \widetilde{\bf G}(d_2) & \cdots & \widetilde{\bf G}(d_p) \\
                                               \bar{\bf G}(d_1) & {\bf 0} & \cdots & {\bf 0}  \\
                                               {\bf 0} & \bar{\bf G}(d_2) & \cdots & {\bf 0} \\
                                               \vdots & \vdots & \ddots & \vdots \\
                                               {\bf 0} & {\bf 0} & \cdots & \bar{\bf G}(d_p)
                           \end{array}\right],
$$
where ${\bf I}={\bf I}_{h+1}$ and ${\bf 0}={\bf 0}_{h+1}$ are, respectively,  $(h+1)\times(h+1)$ unit and zero  matrices.
Note that the first matrix in the product above is of full rank and can be written as
$$
\left[\begin{array}{ccccc}
1 & 0 & 0 & \cdots & 0 \\
0 & 1 & 1 & \cdots & 1 \\
0 & d_1 & d_2 & \cdots & d_p \\
\vdots & \vdots & \vdots & \ddots & \vdots \\
0 & d_1^i & d_2^i & \cdots & d_p^i \\
\vdots & \vdots & \vdots & \ddots & \vdots
\end{array}\right]\,\otimes\,{\bf I}_{h+1}.
$$ Therefore \eqref{linrow} is equivalent to
\bel{abeq}
\left[\begin{array}{cccc} \widetilde{\bf G}(d_1) & \widetilde{\bf G}(d_2) & \cdots  & \widetilde{\bf G}(d_p) \\
                                               \bar{\bf G}(d_1) & {\bf 0} & \cdots  & {\bf 0} \\
                                               {\bf 0} & \bar{\bf G}(d_2) & \cdots  & {\bf 0} \\
                                               \vdots & \vdots & \ddots & \vdots \\
                                               {\bf 0} & {\bf 0} & \cdots  & \bar{\bf G}(d_p)
                           \end{array}\right]\,\c=(1,\,0,\,\ldots,\,0)^T\in \R^{(p+1)(h+1)}.
\ee

Assume that we prove that $(h+1)\times (h+1)$ matrices $\bar{\bf G}(d_m)=0$, $m=1,\ldots,p$, are singular. Note that  $d\,[{\bf H}_{21}(d),\,{\bf H}_{22}(d)]={\bf G}(d)$ due to \eqref{bfG}. Therefore, the definition \eqref{S} of ${\bf S}$ implies that \eqref{abeq} is equivalent to ${\bf S}\,\c=(1,\,0,\,\ldots,\,0)\in\R^{ph+h+1}$. It is obtained from \eqref{abeq} by deleting all rows determined through first rows of matrices $\bar{\bf G}(d_m)$, $m=1,\ldots,p$. And the equation ${\bf S}\,\c=(1,0,\ldots,0)$ follows by ASSUMPTION II and the definition of $\c$.

Consequently, it suffices to show that $\det\,\bar{\bf G}(d_m)=0$, $m=1,\ldots,p$. That is, we need to check that
$$
0=\det\,\left[\begin{array}{cc} {\bf H}_{11}(d_m) & {\bf H}_{12}(d_m) \\
{\bf H}_{21}(d_m) & {\bf H}_{22}(d_m)\end{array}\right]
$$
for any $m=1,\ldots,p$.

Note that with $d=d_m$ the right hand side can be written as
$$
\det\,{\bf H}_{22}(d)\,\det\left[{\bf H}_{11}(d)-{\bf H}_{12}(d){\bf H}_{22}^{-1}(d){\bf H}_{21}(d)\right]
$$
and
\bel{h22}
\det\,{\bf H}_{22}(d)=\prod_{i=1}^s\,\det\,{\bf H}_{m_i}(d).
\ee
Since ${\bf H}_m(d)$ can be decomposed as
\bel{Hmdec}
{\bf H}_m(d)={\bf D}^{-1}_m{\bf R}_m{\bf D}_m,
\ee
where ${\bf D}_m=\mathrm{diag}(1,\,d,\,d^2,\,\ldots,\,d^{m-1})$ and ${\bf R}_m$ is defined in \eqref{Rm} we see that
$$\det\,{\bf H}_m(d)=1+\rho^2+\ldots+\rho^{2m}\ne 0.$$ Now, from \eqref{h22} it follows that $\det\,{\bf H}_{22}\ne 0$.

On the other hand
\bel{shur}
\det\left[{\bf H}_{11}(d)-{\bf H}_{12}(d)\,{\bf H}_{22}^{-1}(d)\,{\bf H}_{21}(d)\right]=(N-1)\alpha(\rho,\,d)+1-\rho^2
-\alpha^2(\rho,\,d)\,\sum_{j=1}^s\1^T\,{\bf H}_{m_j}^{-1}\,\1,
\ee
where $\alpha(\rho,\,d)=1+\rho^2-(d+d^{-1})\rho$.

The decomposition \eqref{Hmdec} of ${\bf H}_m$ gives
$$
\1^T\,{\bf H}_m^{-1}\,\1=\mathrm{tr}(\1^T\,{\bf D}_m^{-1}\,{\bf R}_m^{-1}\,{\bf D}_m\,\1)=\mathrm{tr}({\bf D}_m\1\1^T{\bf D}_m^{-1}{\bf R}_m^{-1})
$$
Moreover, since $\mathrm{tr}({\bf A})=\mathrm{tr}({\bf A}^T)$
$$
\1^T\,{\bf H}_m^{-1}\,\1=\mathrm{tr}(({\bf D}_m\1\1^T{\bf D}_m^{-1}{\bf R}_m^{-1})^T)=\mathrm{tr}({\bf R}_m^{-1}{\bf D}_m^{-1}\1\1^T{\bf D}_m)
=\mathrm{tr}({\bf D}_m^{-1}\1\1^T{\bf D}_m{\bf R}_m^{-1}).
$$
Combining the last two expressions for $\1^T\,{\bf H}_m^{-1}\,\1$ we get
$$
\1^T\,{\bf H}_m^{-1}\,\1=\mathrm{tr}\left(\tfrac{1}{2}({\bf D}_m\1\1^T{\bf D}_m^{-1}+{\bf D}_m^{-1}\1\1^T{\bf D}_m){\bf R}_m^{-1}\right).
$$

Note that
$$
({\bf D}_m\1\1^T{\bf D}_m^{-1}+{\bf D}_m^{-1}\1\1^T{\bf D}_m)_{ij}=d^{|i-j|}+d^{-|i-j|},
$$
and that
$$
\tfrac{1}{2}(d^k+d^{-k})=T_k(\tfrac{1}{2}(d+d^{-1})),\qquad k=0,1,\ldots,
$$
where $(T_k)$ is the $k$th Chebyshev polynomials of the first type.

Thus
$$
\1^T\,{\bf H}_m^{-1}\,\1=\mathrm{tr}\,{\bf T}_m(x)\,{\bf R}_m^{-1},
$$
where $x=x(d)=\tfrac{1}{2}(d+d^{-1})$ and the matrix ${\bf T}_m$ is defined in \eqref{Tm}. Plugging this expression to \eqref{shur} we find out that
$$
\det({\bf H}_{11}(d)-{\bf H}_{12}(d)\,{\bf H}_{22}^{-1}(d)\,{\bf H}_{21}(d))=Q_p(x(d)),
$$
where $Q_p$ is the polynomial defined in \eqref{Qp}. By ASSUMPTION I $Q_p(x(d_m))=0$, thus the above equality gives $\det\,\bar{\bf G}(d_m)=0$, $m=1,\ldots,p$. Finally, we conclude that the constraints \eqref{unb} and \eqref{pat} are satisfied and thus the proof of  point {\bf 2.} is completed.

{\bf Ad. 3.} First, we will show that for $\overline{\r}$ defined by \eqref{ery} the identity \eqref{idp} holds. To this end observe that by \eqref{zerid} for $i=p$, \eqref{invD} and \eqref{Dw}
$$
\Le^{p+1}\left(\Id-\sum_{m=1}^p\,a_m\Ri^m\right)\overline{\w}=\Le\left(\Le^p-\sum_{m=1}^p\,a_m\,\Le^{p-m}\right)\D^{-1}\overline{\underline{\Lambda}}
=\Le\D^{-1}\left(\Le^p-\sum_{m=1}^p\,a_m\,\Le^{p-m}\right)\overline{\underline{\Lambda}}.
$$
Note also that for any $j=1,\ldots,p$ by \eqref{secondid}
$$
\left(\Le^p-\sum_{m=1}^p\,a_m\Le^{p-m}\right)\overline{d}_j=v_p(d_j)\overline{d}_j.
$$
By the definition \eqref{ak} of $a_m$, $m=1,\ldots,p$ it follows that $v_p(d_j)=0$. Due to the definition of $\Lambda$ through \eqref{lamb} we conclude that $\Le^{p+1}\overline{\r}=\overline{\0}$.

In order to check \eqref{idrl} first we note that due to \eqref{invD} it follows from \eqref{Ly} and \eqref{rilj} that for $\overline{y}=(y^n)_{n\ge 0}$ and $\overline{\x}=(\x,\,\x,\,\ldots)$
$$
\D^{-1}\,\overline{y}\,\overline{\x}=(\Id-{\bf C}^T\Ri)\,{\bf N}(y)\,\overline{y}\,\overline{\x}.
$$
Therefore for any $i\ge 0$ any $d_j$ and $\e_{j_k}$ by \eqref{zerid}
$$
\Le^i\left(\Id-\sum_{m=1}^p\,a_m\Ri^m\right)\,\D^{-1}\,\overline{d}_j\,\overline{\e}_{j_k}
$$$$=\left(\Le^p-\sum_{m=1}^p\,a_m\,\Le^{p-m}\right)\,\Ri^{p-i}\,\overline{d}_j\,{\bf N}(d_j)\overline{\e}_{j_k}
-\left(\Le^p-\sum_{m=1}^p\,a_m\,\Le^{p-m}\right)\,\Ri^{p-(i-1)}\,\overline{d}_j\,{\bf C}^T\,{\bf N}(d_j)\overline{\e}_{j_k}.
$$

Finally, we use \eqref{firstid} with $\overline{y}=\overline{d}_j$, $\overline{\x}={\bf N}(d_j)\overline{\e}_{j_k}$ to the first part and with $\overline{y}=\overline{d}_j$, $\overline{\x}={\bf C}^T\,{\bf N}(d_j)\overline{\e}_{j_k}$ to the second part of the expression at the right hand side of the equation above arriving at
$$(\Id-\Ri\,\Le)\,\Le^i\left(\Id-\sum_{m=1}^p\,a_m\Ri^m\right)\,\D^{-1}\,\overline{d}_j\,\overline{\e}_{j_k}
=\left(v_i(d_j){\bf I}-v_{i-1}(d_j){\bf C}^T\right)\,{\bf N}(d_j)\,(\e_{j_k},\,\0,\,\0,\,\ldots).
$$
Thus \eqref{idrl} holds true.

Finally we will prove the formula \eqref{waria} for the variance of the BLUE $\hat{\mu}_t$. To this end we observe first that
$$
\cov(\hat{\mu}_t,\X_{t-i})=\w_i+\sum_{k=1}^{N-1}\,{\bf C}^k\w_{i+k}+\sum_{k=1}^{i\wedge (N-1)}({\bf C}^T)^k\w_{i-k}
$$
for any $i=0,1,\ldots$ On the other hand, due to \eqref{Dw}, we see that the right hand side of the above equality is equal to $\underline{\Lambda}_i$.  That is, for any $i=0,1,\ldots$
$$
\cov(\hat{\mu}_t,\X_{t-i})=\sum_{j\in H'}\,\lambda_{j,i}\,\e_j.
$$
Now, we write
$$
\var\,\hat{\mu}_t=\sum_{i=0}^{\infty}\w_i^T\cov(\hat{\mu}_t,\X_{t-i})=\sum_{i=0}^{\infty}\,\sum_{j\in H'}\,\lambda_{j,i}\w_i^T\e_j.
$$
Due to the constraints \eqref{unb} and \eqref{pat} it follows from the above formula that $\var\,\hat{\mu}_t=\lambda_{0,0}$. Thus, \eqref{waria} follows from \eqref{lamb}.
\end{proof}

\end{document}